\documentclass[a4paper,12pt]{article}
\usepackage{pdfsync}
\usepackage[francais]{babel}
\usepackage[T1]{fontenc}
\usepackage[utf8]{inputenc}
\usepackage{amssymb}
\usepackage{amsmath}
\usepackage{amsthm}
\usepackage{url}

\newtheorem{theorem}{Théorème}
\newtheorem{proposition}{Proposition}
\newtheorem{corollary}{Corollaire}
\newtheorem{definition}{Définition}

\title{Un théorème du support pour la fibration de Hitchin}
\author{Pierre-Henri Chaudouard\footnote{soutenu par l'Institut Universitaire de France et les  projets Ferplay ANR-13-BS01-0012 et  Vargen ANR-13-BS01-0001-01 de l'ANR}  \ et Gérard Laumon}
\date{}
\begin{document}
\maketitle

\section{Introduction}

Soit $X$ une courbe connexe, projective et lisse, de genre $g$, sur un corps $k$ algébriquement clos. Soit $D$ un diviseur sur $X$ de degré $>2g-2$ et soit $G$ un groupe réductif. Notons $f:\mathcal{N}\rightarrow\mathbb{A}$ la fibration de Hitchin associée à ces données.
\medskip

Sur l'ouvert elliptique $\mathbb{A}^{\mathrm{ell}}\subset\mathbb{A}$, chaque composante connexe de cette fibration est une gerbe sur un morphisme propre $f^{\mathrm{ell}}:N^{\mathrm{ell}}\rightarrow\mathbb{A}^{\mathrm{ell}}$, de source $N^{\mathrm{ell}}$ lisse sur $k$, auquel on peut appliquer le théorème de décomposition de Beilinson, Bernstein, Deligne et Gabber \cite{Beilinson}. L'outil principal dans la démonstration par Ngô \cite{Ngo} du lemme fondamental de Langlands-Shelstad pour $G$, est un théorème sur les supports des constituents simples de la cohomologie relative de $f^{\mathrm{ell}}$.
\medskip

Dans \cite{Chaudouard}, nous avons prolongé ce résultat de Ngô à l'ouvert génériquement régulier semi-simple $\mathbb{A}^{\mathrm{ell}}\subset\mathbb{A}^{\mathrm{grss}}\subset\mathbb{A}$ de la fibration de Hitchin.
\medskip

Dans le cas particulier où $k$ est de caractéristique nulle, $G=\mathrm{GL}(n)$ et $D$ est de degré $>2g-2$, le théorème de Ngô sur $\mathbb{A}^{\mathrm{ell}}$ et notre extension à $\mathbb{A}^{\mathrm{grss}}$ disent que la cohomologie relative de $f^{\mathrm{grss}}:N^{\mathrm{grss}}\rightarrow\mathbb{A}^{\mathrm{grss}}$ (pour un schéma $N^{\mathrm{grss}}\supset N^{\mathrm{ell}}$ convenablement défini à partir de la restriction de $\mathcal{N}$ à $\mathbb{A}^{\mathrm{grss}}$) est complètement déterminée par sa restriction à n'importe quel ouvert non vide de $\mathbb{A}^{\mathrm{grss}}$.
\medskip

L'objet de cette note est de prolonger ce dernier résultat à $\mathbb{A}$ tout entier.

\section{La fibration de Hitchin}

Soient $n$ un entier positif (le rang) et $e$ un entier premier à $n$ (le degré).
\medskip

Soit $k$ un corps algébriquement clos, soit $X$ une courbe connexe, projective et lisse sur $k$, de genre $g$, et soit $D=\sum_{x}d_{x}[x]$ un diviseur sur $X$ que l'on suppose effectif ($d_{x}\geq 0,~\forall x$) et de degré
$$
d=\sum_{x}d_{x}>2g-2
$$
(voir la section \ref{Canonique} pour des commentaires sur le cas où $D$ est un diviseur canonique et donc $d=2g-2$).
\medskip

Soit enfin $\ell$ un nombre premier inversible dans $k$ (pour la cohomologie $\ell$-adique).
\medskip

Rappelons \cite{Hitchin} qu'un fibré de Hitchin est un couple $(\mathcal{E},\theta)$ formé d'un fibré vectoriel $\mathcal{E}$ sur $X$, de rang $n$ et de degré $e$, et d'un endomorphisme tordu $\theta :\mathcal{E}\rightarrow \mathcal{E}(D):=\mathcal{E}\otimes_{\mathcal{O}_{X}}\mathcal{O}_{X}(D)$ de ce fibré vectoriel. Rappelons aussi que le champ $\mathcal{N}_{n}^{e}$ des fibrés de Hitchin est algébrique et localement de type fini sur $k$.
\medskip

La base de la fibration de Hitchin est le schéma affine
$$
\mathbb{A}_{n}=\bigoplus_{i=1}^{n}H^{0}(X,\mathcal{O}_{X}(iD));
$$
sa dimension $d_{\mathbb{A}_{n}}$ peut être calculée par le théorème de Riemann-Roch et est égale à
$$
d_{\mathbb{A}_{n}}=n(1-g)+\frac{n(n+1)}{2}d.
$$

La fibration de Hitchin est le morphisme de champs algébriques
\begin{equation*}
f_{n}:\mathcal{N}_{n}^{e}\rightarrow \mathbb{A}_{n}
\end{equation*}
qui envoie $(\mathcal{E},\theta)$ sur le polynôme caractéristique de $\theta$
$$
a=(-\mathrm{tr}(\theta),\dots,(-1)^{n}\mathrm{det}(\theta)).
$$

\section{Courbes spectrales}

Soit
$$
p:\Sigma=\mathbb{V}(\mathcal{O}_{X}(-D))\rightarrow X
$$
l'espace total du fibré en droites $\mathcal{O}_{X}(D)$ et $u$ la section universelle de $p^{\ast}\mathcal{O}_{X}(D)$. Tout $a\in\mathbb{A}_{n}$ définit une section globale
$$
P_{a}(u)=u^{n}+p^{\ast}(a_{1})u^{n-1}+\cdots+p^{\ast}(a_{n})\in H^{0}(\Sigma,p^{\ast}\mathcal{O}_{X}(nD)).
$$
La courbe spectrale
$$
X_{a}=X_{n,a}\subset\Sigma
$$
est le diviseur de Cartier des zéros de cette section $P_{a}(u)$. Un théorème de Bertini assure que $X_{a}$ est connexe, mais elle n'est pas nécessairement irréductible, ni réduite. La restriction $\pi_{a}:X_{a}\rightarrow X$ de $p$ à $X_{a}$ est un morphisme fini et plat de degré $n$, et on a
$$
\pi_{a,\ast}\mathcal{O}_{X_{a}}=\mathcal{O}_{X}\oplus\mathcal{O}_{X}(-D)\oplus\cdots\oplus\mathcal{O}_{X}(-(n-1)D),
$$
de sorte que
$$
\chi(X_{a},\mathcal{O}_{X_{a}})=n(1-g)-\frac{n(n-1)}{2}d.
$$

Pour $a\in\mathbb{A}_{n}$, soit $X_{a,\eta}$ la fibre de $\pi_{a}:X_{a}\rightarrow X$ au point générique $\eta$ de $X$ et $j_{a}:X_{a,\eta}\hookrightarrow X_{a}$ l'inclusion; $X_{a,\eta}$ est une réunion disjointe finie de schémas artiniens, un pour chaque point générique de $X_{a}$.
\medskip

Suivant la définition donnée par Schaub \cite{Schaub}, un module sans torsion de rang $1$ sur $X_{a}$ est un $\mathcal{O}_{X_{a}}$-module cohérent $\mathcal{F}$ tel que l'homomorphisme canonique
$$
\mathcal{F}\rightarrow j_{a,\ast}j_{a}^{\ast}\mathcal{F}
$$
est injectif et tel qu'en chaque point générique de $X_{a}$, les fibres de $\mathcal{F}$ et $\mathcal{O}_{X_{a}}$ ont la même longueur.
\medskip

Si $X_{a}$ est réduite, c'est la notion habituelle; en particulier, $\mathcal{F}$ est localement libre de rang $1$ sur le lieu lisse de $X_{a}$.
\medskip

Si $\mathcal{F}$ est un module sans torsion de rang $1$ sur $X_{a}$,
$$
{\mathcal{E}_{\mathcal{F}}=\pi_{a,\ast}\mathcal{F}}
$$
est un fibré vectoriel de rang $n$ et la multiplication par $u$ le munit d'un endomorphisme tordu ${\theta_{\mathcal{F}}:\mathcal{E}_{\mathcal{F}}\rightarrow\mathcal{E}_{\mathcal{F}}(D)}$. On vérifie à l'aide du théorème de Riemann-Roch que
$$
\mathrm{deg}(\mathcal{F}):=\chi(X_{a},\mathcal{F})-\chi(X_{a},\mathcal{O}_{X_{a}})=\mathrm{deg}(\mathcal{E}_{\mathcal{F}})+\frac{n(n-1)}{2}d.
$$

\begin{theorem}[Hitchin \cite{Hitchin} si $X_{a}$ est lisse; Beauville, Narasimhan et Ramanan \cite{Beauville} si $X_{a}$ est intègre; Schaub \cite{Schaub} en général]
Le foncteur $\mathcal{F}\mapsto(\mathcal{E}_{\mathcal{F}},\theta_{\mathcal{F}})$ est un isomorphisme du champ modulaire des modules sans torsion de rang $1$ et degré $e+\frac{n(n-1)}{2}d$ sur $X_{a}$, sur le champ algébrique $f_{n}^{-1}(a)\subset\mathcal{N}_{n}^{e}$.
\end{theorem}

\section{Symétries}

Soit $\mathcal{P}_{n}$ la composante de degré $0$ du champ de Picard relatif de la courbe spectrale universelle (le diviseur de Cartier relatif dans $\mathbb{A}_{n}\times_{k}\Sigma/\mathbb{A}_{n}$ défini par l'équation $P_{a}(u)=0$): pour chaque $a\in\mathbb{A}_{n}$, $\mathcal{P}_{n,a}$ est le champ de Picard des modules inversibles de degré $0$ sur $X_{a}$. 
\medskip

Ce champ de Picard $\mathcal{P}_{n}$ agit sur $\mathcal{N}_{n}^{e}$: si $\mathcal{L}\in\mathcal{P}_{n,a}$ et $\mathcal{F}$ est un module sans torsion de rang $1$ sur $X_{a}$, l'action est définie par
$$
\mathcal{L}\cdot(\mathcal{E}_{\mathcal{F}},\theta_{\mathcal{F}})
=(\mathcal{E}_{\mathcal{L}\otimes\mathcal{F}},\theta_{\mathcal{L}\otimes\mathcal{F}}).
$$

Le champ algébrique $\mathcal{P}_{n}$ n'est ni séparé, ni de type fini, mais sa composante neutre $\mathcal{J}_{n}\subset\mathcal{P}_{n}$ admet un espace grossier $J_{n}$ qui est un schéma en groupes, séparé, lisse et de type fini sur $\mathbb{A}_{n}$ \cite{Bosch}. Pour chaque $a\in\mathbb{A}_{n}$, $J_{n,a}$ est le schéma en groupes des classes d'isomorphie de fibrés inversibles sur $X_{a}$ dont la restriction à chaque composante irreducible de $X_{a}$ est de degré $0$. 
\medskip

Soit $\mathbb{A}_{n}^{\mathrm{lisse}}\subset\mathbb{A}_{n}$ l'ouvert dense au-dessus duquel la courbe spectrale est lisse. La restriction $J_{n}^{\mathrm{lisse}}=J_{n}|\mathbb{A}_{n}^{\mathrm{lisse}}$ est un schéma abélien. La restriction $\mathcal{N}_{n}^{e}|\mathbb{A}_{n}^{\mathrm{lisse}}$ admet un espace de module grossier $f_{n}^{\mathrm{lisse}}: N_{n}^{e,\mathrm{lisse}}\rightarrow\mathbb{A}_{n}^{\mathrm{lisse}}$ qui est un torseur sous $J_{n}^{\mathrm{lisse}}$. Les morphismes $f_{n}^{\mathrm{lisse}}$ et $J_{n}^{\mathrm{lisse}}\rightarrow\mathbb{A}_{n}^{\mathrm{lisse}}$ sont tous les deux lisses, purement de dimension relative
$$
d_{f_{n}}=n(g-1)+\frac{n(n-1)d}{2}+1.
$$
Par suite le faisceau $\ell$-adique $R^{1}f_{n,\ast}^{\mathrm{lisse}}\mathbb{Q}_{\ell}$ est un système local de rang $2d_{f_{n}}$ et
$$
R^{i}f_{n,\ast}^{\mathrm{lisse}}\mathbb{Q}_{\ell}=\bigwedge^{i}R^{1}f_{n,\ast}^{\mathrm{lisse}}\mathbb{Q}_{\ell},~\forall i=0,1,\dots,2d_{f_{n}}.
$$

\section{Rappels sur les faisceaux pervers}

Soit $A$ un schéma de type fini sur $k$. Dans la catégorie dérivée $D_{\mathrm{c}}^{\mathrm{b}}(A,\mathbb{Q}_{\ell})$ des faisceaux $\ell$-adiques sur $A$, on a la sous-catégorie pleine des faisceaux pervers $\mathrm{Perv}(A,\mathbb{Q}_{\ell})$ définie par Beilinson, Bernstein, Deligne et Gabber \cite{Beilinson}.
\medskip

La catégorie $\mathrm{Perv}(A,\mathbb{Q}_{\ell})$ est abélienne. La dualité de Poincaré sur $D_{\mathrm{c}}^{\mathrm{b}}(A,\mathbb{Q}_{\ell})$ envoie $\mathrm{Perv}(A,\mathbb{Q}_{\ell})$ dans elle-même. On a des foncteurs de cohomologie ${}^{\mathrm{p}}\mathcal{H}^{i}:D_{\mathrm{c}}^{\mathrm{b}}(A,\mathbb{Q}_{\ell})\rightarrow\mathrm{Perv}(A,\mathbb{Q}_{\ell})$ qui vérifient les propriétés usuelles (suite exacte longue de cohomologie, ...).
\medskip

Tous les objets de $\mathrm{Perv}(A,\mathbb{Q}_{\ell})$ sont de longueur finie.
Les objets simples sont les faisceaux pervers de la forme
$$
{i_{a,\ast}\mathrm{IC}_{\overline{\{a\}},\mathcal{L}}[\mathrm{d}_{a}]}
$$
où $a\in A$ (un point au sens de Zariski, non fermé en général),
$$
i_{a}:\overline{\{a\}}\hookrightarrow A
$$
est l'adhérence de Zariski de $a$ dans $A$, $\mathrm{d}_{a}$ est la dimension de $\overline{\{a\}}$, $\mathcal{L}$ est un système local $\ell$-adique irréductible sur $\{a\}$ qui se prolonge à un voisinage ouvert de $a$ dans $\overline{\{a\}}$, et $\mathrm{IC}_{\overline{\{a\}},\mathcal{L}}$ est le complexe d'intersection de $\overline{\{a\}}$ à valeurs dans $\mathcal{L}$ (en particulier, $\mathrm{IC}_{\overline{\{a\}},\mathcal{L}}|\{a\}=\mathcal{L}[0]$).
\medskip

Tout faisceau pervers semi-simple $K$ sur $A$ admet donc une décomposition
$$
K\cong \bigoplus_{a\in A}i_{a,\ast}\mathrm{IC}_{\overline{\{a\}},\mathcal{L}_{a}}[\mathrm{d}_{a}]
$$
où pour chaque $a$, $\mathcal{L}_{a}$ est un système local $\ell$-adique semi-simple sur $\{a\}$ qui se prolonge à un voisinage ouvert de $a$ dans $\overline{\{a\}}$. 
\medskip

Un complexe $K\in D_{\mathrm{c}}^{\mathrm{b}}(A,\mathbb{Q}_{\ell})$ est dit semi-simple si
$$
K\cong\bigoplus_{i}{}^{\mathrm{p}}\mathcal{H}^{i}K[-i]
$$
et si chaque ${}^{\mathrm{p}}\mathcal{H}^{i}K$ est un faisceau pervers semi-simple. Un tel complexe admet une décomposition
$$
K\cong\bigoplus_{a\in A}\bigoplus_{i}i_{a,\ast}\mathrm{IC}_{\overline{\{a\}},\mathcal{L}_{a}^{i}}[\mathrm{d}_{a}-i]
$$
où pour chaque $a$ et chaque $i$, $\mathcal{L}_{a}^{i}$ est un système local $\ell$-adique semi-simple sur $\{a\}$ qui se prolonge à un voisinage ouvert de $a$ dans $\overline{\{a\}}$. On note alors
$$
{\mathrm{Socle}(K)=\{a\in A\mid\exists i\hbox{ tel que }\mathcal{L}_{a}^{i}\not=(0)\}},
$$
et pour chaque $a\in\mathrm{Socle}(K)$,
$$
n_{a}^{+}(K)=\mathrm{Sup}\{i\mid \mathcal{L}_{a}^{i}\not=(0)\}\hbox{ et }n_{a}^{-}(K)=\mathrm{Inf}\{i\mid \mathcal{L}_{a}^{i}\not=(0)\}.
$$
L'amplitude en $a$ de $K$ est l'entier
$$
\mathrm{Amp}_{a}(K)=n_{a}^{+}(K)-n_{a}^{-}(K)
$$
Si $K$ est auto-dual, on a $n_{a}^{+}(K)=-n_{a}^{-}(K)$ et donc $\mathrm{Amp}_{a}(K)=2n_{a}^{+}(K)$.
\smallskip

Un résultat fondamental de la théorie des faisceaux pervers est le théorème de décomposition. 

\begin{theorem}[Beilinson, Bernstein, Deligne et Gabber \cite{Beilinson}]\label{decomposition}
Soient $N$ un schéma lisse purement de dimension $d_{N}$ sur $k$ et $f:N\rightarrow A$ un morphisme propre. Alors le complexe auto-dual $Rf_{\ast}^{\mathrm{ell}}\mathbb{Q}_{\ell}[d_{N}]$ est semi-simple.
\end{theorem}

\section{Le théorème de Ngô sur le lieu elliptique}

Le lieu elliptique est l'ouvert $\mathbb{A}_{n}^{\mathrm{ell}}\subset\mathbb{A}_{n}$ au-dessus duquel $X_{a}$ est intègre; il contient le lieu lisse.
\medskip

Nous savons d'après les travaux de Altmann et Kleiman \cite{Altmann}, qu'au-dessus du lieu elliptique, l'espace grossier du champ $\mathcal{N}_{n}^{e}$ est un schéma $N_{n}^{e,\mathrm{ell}}$ sur lequel le schéma en groupes lisse $J_{n}^{\mathrm{ell}}=J_{n}|\mathbb{A}_{n}^{\mathrm{ell}}$ agit. De plus, le schéma $N_{n}^{e,\mathrm{ell}}$ est lisse de dimension $d_{N_{n}^{e,\mathrm{ell}}}=d_{\mathbb{A}_{n}}+d_{f_{n}}=n^{2}d+1$ sur $k$, et la fibration de Hitchin $f_{n}^{\mathrm{ell}}:N_{n}^{e,\mathrm{ell}}\rightarrow \mathbb{A}_{n}^{\mathrm{ell}}$ est plate, projective, à fibres connexes de dimension $d_{f_{n}}$.

Comme $N_{n}^{e,\mathrm{ell}}$ est lisse sur $k$ et $f_{n}^{\mathrm{ell}}:N_{n}^{e,\mathrm{ell}}\rightarrow\mathbb{A}_{n}^{\mathrm{ell}}$ est propre, le complexe auto-dual $Rf_{n,\ast}^{\mathrm{ell}}\mathbb{Q}_{\ell}[d_{N_{n}^{e,\mathrm{ell}}}]$ est semi-simple d'après le théorème de décomposition.
\medskip

Dans le cas particulier considéré, le théorème cohomologique principal de Ngô dans sa preuve du lemme fondamental de Langlands-Shelstad s'énonce:

\begin{theorem}[Ngô \cite{Ngo}]\label{support}
Si $k$ est de caractéristique nulle, le socle du complexe auto-dual $Rf_{n,\ast}^{\mathrm{ell}}\mathbb{Q}_{\ell}[d_{N_{n}^{e,\mathrm{ell}}}]$ est réduit au point générique de $\mathbb{A}_{n}$. En d'autres termes on a un isomorphisme
$$
{Rf_{n,\ast}^{\mathrm{ell}}\mathbb{Q}_{\ell}\cong\bigoplus_{i}\mathrm{IC}_{\mathbb{A}_{n}^{\mathrm{ell}},R^{i}f_{n,\ast}^{\mathrm{lisse}}\mathbb{Q}_{\ell}}[-i]}.
$$
\end{theorem}

Les ingrédients principaux de la preuve sont:
\medskip

\begin{itemize}
\item l'action de $J_{n}^{\mathrm{ell}}$ sur la partie elliptique de la fibration de Hitchin est suffisamment libre pour borner inférieurement l'amplitude en $a$ de $Rf_{n,\ast}^{\mathrm{ell}}\mathbb{Q}_{\ell}[d_{N_{n}^{e,\mathrm{ell}}}]$, et donc supérieurement la codimension de $a$ dans $\mathbb{A}_{n}^{\mathrm{ell}}$, pour chaque $a\in \mathrm{Socle}(Rf_{n,\ast}^{\mathrm{ell}}\mathbb{Q}_{\ell}[d_{N_{n}^{e,\mathrm{ell}}}])$.
\medskip

\item une inégalité de type Severi qui borne inférieurement la codimension de tout point dans $\mathrm{Socle}(Rf_{n,\ast}^{\mathrm{ell}}\mathbb{Q}_{\ell}[d_{N_{n}^{e,\mathrm{ell}}}])$ (pour le moment, cette inégalité n'est disponible en toute généralité qu'en caractéristique nulle).
\end{itemize}
\medskip

Plus précisément, soient $A$ un schéma séparé de type fini sur $k$ et soit $J$ un schéma en groupes commutatifs, lisse, de type fini et à fibres connexes sur $A$.

Pour chaque $a\in A$ et chaque point géométrique $\overline{a}\rightarrow a$ on a un dévissage
$$
{0\rightarrow J_{\overline{a}}^{\mathrm{aff}}\rightarrow J_{\overline{a}}\rightarrow J_{\overline{a}}^{\mathrm{ab}}\rightarrow 0}
$$
où $J_{\overline{a}}^{\mathrm{ab}}$ est une variété abélienne et $J_{\overline{a}}^{\mathrm{aff}}$ est affine. Les dimensions de $J_{\overline{a}}^{\mathrm{aff}}$ et $J_{\overline{a}}^{\mathrm{ab}}$ dépendent seulement de $a$ et sont notées dans la suite ${\mathrm{d}_{a}^{\mathrm{aff}}(J)}$ et ${\mathrm{d}_{a}^{\mathrm{ab}}(J)}$.

Soit $f:N\rightarrow A$ un $A$-schéma, muni d'une action de $J$.

\begin{theorem}[Ngô \cite{Ngo}]\label{technique}
On suppose que $f:N\rightarrow A$ est projectif, purement de dimension relative $\mathrm{d}_{f}$, que $N$ est lisse sur $k$ (purement de dimension $d_{N}=d_{A}+d_{f}$), que les stabilisateurs dans $J$ des points fermés dans $N$ sont tous affines, et que le module de Tate $V_{\ell}(J)$ est polarisable. Alors pour tout $a\in\mathrm{Socle}(Rf_{\ast}\mathbb{Q}_{\ell}[d_{N}])$, on a
$$
2d_{N}-d_{f}+d_{a}\geq\frac{1}{2}\mathrm{Amp}_{a}(Rf_{\ast}\mathbb{Q}_{\ell}[d_{N}])\geq\mathrm{d}_{a}^{\mathrm{ab}}(J),
$$
soit encore
$$
\mathrm{d}_{f}-\mathrm{d}_{A}+\mathrm{d}_{a}\geq n_{a}^{+}(Rf_{\ast}\mathbb{Q}_{\ell}[d_{N}])\geq \mathrm{d}_{a}^{\mathrm{ab}}(J).
$$
\end{theorem}
\medskip

Pour tout $a\in\mathbb{A}_{n}^{\mathrm{ell}}$, la courbe spectrale $X_{a}$ est intègre et on peut donc introduire sa normalisation $\widetilde{X}_{a}$. On note
$$
\delta_{X_{a}}=\mathrm{long}(\mathcal{O}_{\widetilde{X}_{a}}/\mathcal{O}_{X_{a}}).
$$
On a 
$$
d_{f_{n}}-\mathrm{d}_{a}^{\mathrm{ab}}(J_{n})=\mathrm{d}_{a}^{\mathrm{aff}}(J_{n})=\delta_{X_{a}}
$$
puisque la jacobienne de $\widetilde{X}_{a}$ est la partie abélienne de la composante neutre $J_{n,a}$ du schéma de Picard de $X_{a}$.

\begin{theorem}[Diaz and Harris \cite{Diaz}]\label{severi}
Si $k$ est de caractéristique nulle, pour tout $a\in\mathbb{A}_{n}^{\mathrm{ell}}$, on a l'inégalité de Severi
$$
\mathrm{d}_{\mathbb{A}_{n}}-\mathrm{d}_{a}\geq\delta_{X_{a}}=\mathrm{d}_{a}^{\mathrm{aff}}(J_{n}),
$$
ou de manière équivalente
$$
{\mathrm{d}_{a}^{\mathrm{ab}}(J_{n})\geq d_{f_{n}}-d_{\mathbb{A}_{n}}+d_{a}}.
$$
\end{theorem}

La preuve du théorème \ref{support} de Ngô consiste à combiner les inégalités de dimension formulées ci-dessus:

\begin{proof}
Soit $a\in\mathrm{Socle}(Rf_{n,\ast}^{\mathrm{ell}}\mathbb{Q}_{\ell}[d_{N_{n}^{e,\mathrm{ell}}}])$. En combinant le théorème \ref{technique} et l'inégalité de Severi (théorème \ref{severi}), on obtient
$$
d_{a}^{\mathrm{ab}}(J_{n}^{\mathrm{ell}})=n_{a}^{+}(Rf_{\ast}^{\mathrm{ell}}\mathbb{Q}_{\ell}[d_{N_{n}^{e,\mathrm{ell}}}])=d_{a}-(d_{\mathbb{A}_{n}}-d_{f_{n}}).
$$

La seconde de ces inégalités implique que $Rf_{n,\ast}^{\mathrm{ell},2\mathrm{d}_{n}}\mathbb{Q}_{\ell}$ admet comme facteur direct un faisceau $\ell$-adique non trivial de support $\overline{\{a\}}$. 

Mais si $a$ n'est pas le point générique de $\mathbb{A}_{n}^{\mathrm{ell}}$, c'est impossible. En effet les fibres de $f_{n}^{\mathrm{ell}}$ étant toutes irréductibles d'après Altmann et Kleiman \cite{Altmann}, on a 
$$
{Rf_{n,\ast}^{\mathrm{ell},2\mathrm{d}_{n}}\mathbb{Q}_{\ell}\cong\mathbb{Q}_{\ell}}.
$$
\end{proof}

\section{Hors du lieu elliptique}

Hors du lieu elliptique, $\mathcal{N}_{n}^{e}$ est plus compliqué:
\medskip

\begin{itemize}
\item même si l'on tue les automorphismes scalaires, $\mathcal{N}_{n}^{e}$ reste hautement non séparé: les groupes d'automorphismes sont affines mais ne sont pas finis en général;
\medskip

\item $\mathcal{N}_{n}^{e}$ est localement de type fini mais n'est plus quasi-compact.
\end{itemize}
\medskip

Pour obtenir un espace plus utilisable, on tronque $\mathcal{N}_{n}^{e}$ à l'aide la stabilité au sens de Mumford, Narasimhan, Seshadri et Hitchin.
\medskip

Tout fibré vectoriel $\mathcal{F}\not=(0)$ a une pente
$$
\mu(\mathcal{F})=\mathrm{deg}(\mathcal{F})/\mathrm{rank}(\mathcal{F}).
$$

\begin{definition}[Hitchin]
Un fibré de Hitchin $(\mathcal{E},\theta)$ est stable si pour tout sous-fibré $(0)\not=\mathcal{F}\subset\mathcal{E}$ tel que $\theta(\mathcal{F})\subset\mathcal{F}(D)$, on a
$$
\mu(\mathcal{F})<\mu(\mathcal{E}).
$$
\end{definition}

\begin{theorem}[Hitchin, Nitsure \cite{Nitsure}]
Les fibrés de Hitchin stables forment un sous-champ ouvert $\mathcal{N}_{n}^{e,\mathrm{st}}$ qui admet un schéma de modules grossier $N_{n}^{e,\mathrm{st}}$.

Le schéma $N_{n}^{e,\mathrm{st}}$ est lisse purement de dimension $d_{\mathbb{A}_{n}}+d_{f_{n}}=n^{2}d+1$ sur $k$. Le morphisme $f_{n}^{\mathrm{st}}:N_{n}^{e,\mathrm{st}}\rightarrow\mathbb{A}_{n}$ est projectif à fibres connexes.
\end{theorem}

\section{Dimension des fibres de la fibration de Hitchin}

Nous avons vu que le morphisme $f_{n}:N_{n}^{e,\mathrm{lisse}}\rightarrow\mathbb{A}_{n}^{\mathrm{lisse}}$ est lisse purement de dimension relative $d_{f_{n}}$. Par suite le morphisme de champs $f_{n}:\mathcal{N}_{n}^{e}\rightarrow\mathbb{A}_{n}$ est purement de dimension relative $d_{f_{n}}-1$ au-dessus de l'ouvert $\mathbb{A}_{n}^{\mathrm{lisse}}$.
\medskip

Il en est de même au-dessus de l'ouvert elliptique et plus généralement au-dessus de l'ouvert $\mathbb{A}_{n}^{\mathrm{grss}}\subset\mathbb{A}_{n}$ où $X_{a}$ est réduite, c'est-à-dire où $\theta$ est génériquement régulier semi-simple. En effet, au-dessus de $\mathbb{A}_{n}^{\mathrm{grss}}$, l'ouvert des modules inversibles sur $X_{a}$ est dense dans le champ $f_{n}^{-1}(a)$ des modules sans torsion de rang générique $1$ sur $X_{a}$, d'après Esteves \cite{Esteves}.
\medskip

Nous démontrerons plus loin (cf. section \ref{preuveDimension}) que :

\begin{proposition}\label{dimension}
Le cône global nilpotent, i.e. la fibre en $0$ de la fibration de Hitchin $\mathcal{N}_{n}^{e}\rightarrow\mathbb{A}_{n}$, est de dimension au plus $d_{f_{n}}-1$.
\end{proposition}

\begin{corollary}
Le morphisme $f_{n}^{\mathrm{st}}:N_{n}^{e,\mathrm{st}}\rightarrow\mathbb{A}_{n}$ est purement de dimension relative $d_{f_{n}}$.
\end{corollary}

\begin{proof}
D'après la proposition, $(f_{n}^{\mathrm{st}})^{-1}(0)$ est de dimension au plus $d_{f_{n}}$, or la dimension des fibres de $f_{n}^{\mathrm{st}}$ ne peut pas chuter par spécialisation.
\end{proof}

\section{Le théorème principal}

Il résulte du théorème de décomposition que $Rf_{n,\ast}^{\mathrm{st}}\mathbb{Q}_{\ell}[d_{\mathbb{A}_{n}}+d_{f_{n}}]$ est semi-simple.
\medskip

Le théorème suivant prolonge le théorème \ref{support} de Ngô, et aussi notre extension du théorème \ref{support} à l'ouvert génériquement régulier semi-simple \cite{Chaudouard}.

\begin{theorem}\label{ThrPr}
Si $k$ est de caractéristique nulle, le socle de $Rf_{n,\ast}^{\mathrm{st}}\mathbb{Q}_{\ell}$ est réduit au point générique de $\mathbb{A}_{n}$. En d'autres termes on a un isomorphisme
$$
{Rf_{n,\ast}^{\mathrm{st}}\mathbb{Q}_{\ell}\cong\bigoplus_{i}\mathrm{IC}_{\mathbb{A}_{n},R^{i}f_{n,\ast}^{\mathrm{lisse}}\mathbb{Q}_{\ell}}[-i]}.
$$
\end{theorem}

\begin{proof}
Soit $\Lambda_{n}$ l'ensemble des couples $(\underline{n},\underline{m})$ de suites $\underline{n}=(n_{1}\geq\cdots\geq n_{s})$ et $\underline{m}=(m_{1},\ldots,m_{s})$ d'entiers strictement positifs telles que
$$
n=n_{1}m_{1}+\cdots+n_{s}m_{s}.
$$

Pour chaque $\lambda=(\underline{n},\underline{m})$ on a un morphisme fini
$$
\iota_{\lambda}:\mathbb{A}_{n_{1}}\times_{k}\cdots\times_{k}\mathbb{A}_{n_{s}}\rightarrow\mathbb{A}_{n}
$$
défini par $P_{\iota_{\lambda}(a_{1},\ldots,a_{s})}(u)=P_{a_{1}}^{m_{1}}(u)\cdots P_{a_{s}}^{m_{s}}(u)$.
\medskip

Pour chaque $a\in\mathbb{A}_{n}$, il existe un unique $\lambda=\lambda(a)\in\Lambda_{n}$ tel que $a$ puisse s'écrire $a=\iota_{\lambda}(a_{1},\ldots,a_{s})$ avec des $a_{i}\in\mathbb{A}_{n_{i}}^{\mathrm{ell}}$ et $P_{a_{i}}(u)\not=P_{a_{j}}(u)$, $\forall i\not=j$. Les $X_{n_{i},a_{i}}$ sont les composantes irréductibles de $X_{a}$ et, pour chaque $i$, $m_{i}$ est la multiplicité de $X_{n_{i},a_{i}}$ dans le diviseur de Cartier $X_{a}$.
\medskip

Les ensembles $\mathbb{A}_{n,\lambda}=\{a\in\mathbb{A}_{n}\mid\lambda(a)=\lambda\}$ forment une stratification de $\mathbb{A}_{n}$ par des parties localement fermées, la strate ouverte étant $\mathbb{A}_{n,((n),(1))}=\mathbb{A}_{n}^{\mathrm{ell}}$ et la strate fermée étant la strate \og{nilpotente\fg} $\mathbb{A}_{n,((1),(n))}$.
\medskip

Pour $\lambda=(\underline{n},\underline{m})\in\Lambda_{n}$ et $a=\iota_{\lambda}(a_{1},\ldots,a_{s})\in\mathbb{A}_{n,\lambda}$ avec des $a_{i}\in\mathbb{A}_{n_{i}}^{\mathrm{ell}}$ comme ci-dessus, on pose $n'=n_{1}+\cdots+n_{s}$ et on définit $a'\in\mathbb{A}_{n'}^{\mathrm{grss}}$ par $P_{a'}(u)=P_{a_{1}}(u)\cdots P_{a_{s}}(u)$, de sorte que $X_{a'}$ n'est autre que la courbe réduite $(X_{a})_{\mathrm{red}}$. 
\medskip

On a un homomorphisme de restriction de $X_{a}$ à $X_{a'}$
$$
J_{n,a}\rightarrow J_{n',a'}
$$
qui est surjectif, et dont le noyau est affine et est une extension successive de copies du groupe additif. De plus on a un homomorphisme 
$$
J_{n',a'}\rightarrow J_{n_{1},a_{1}}\times_{k}\cdots\times_{k}J_{n_{s},a_{s}}
$$
qui est lui aussi surjectif à noyau affine. Par suite
$$
{\mathrm{d}_{a}^{\mathrm{ab}}(J_{n})=\mathrm{d}_{a_{1}}^{\mathrm{ab}}(J_{n_{1}})+\cdots+\mathrm{d}_{a_{s}}^{\mathrm{ab}}(J_{n_{s}})}.
$$

Si $a=\iota_{\lambda}(a_{1},\ldots,a_{s})$ est de plus dans $\mathrm{Socle}(Rf_{n,\ast}^{\mathrm{st}}\mathbb{Q}_{\ell})$, on a d'une part
$$
d_{f_{n}}-d_{\mathbb{A}_{n}}+d_{a}\geq\mathrm{d}_{a}^{\mathrm{ab}}(J_{n})
$$
d'après le théorème \ref{technique}, et d'autre part, pour chaque $i$, on a
$$
\mathrm{d}_{a_{i}}^{\mathrm{ab}}(J_{n_{i}})\geq d_{f_{n_{i}}}-d_{\mathbb{A}_{n_{i}}}+d_{a_{i}}
$$
d'après le théorème \ref{severi} appliqué à $a_{i}\in\mathbb{A}_{n_{i}}^{\mathrm{ell}}$. Or
$$
\mathrm{d}_{a}=\mathrm{d}_{a_{1}}+\cdots+\mathrm{d}_{a_{s}},
$$
$$
d_{f_{n}}-d_{\mathbb{A}_{n}}=n(2g-2-d)+1
$$
et
$$
\mathrm{d}_{a}^{\mathrm{ab}}(J_{n})=\mathrm{d}_{a_{1}}^{\mathrm{ab}}(J_{n_{1}})+\cdots+\mathrm{d}_{a_{s}}^{\mathrm{ab}}(J_{n_{s}}).
$$
Par suite on obtient l'inégalité
$$
{1-s\geq (n-n_{1}-\cdots-n_{s})(d-2g+2)}
$$
qui n'est possible que si $s=1$ et $n_{1}=n$, c'est-à-dire $\mathbb{A}_{n,\lambda}=\mathbb{A}_{n}^{\mathrm{ell}}$ puisque l'on a supposé $d>2g-2$.
\end{proof}

\section{Dimension du cône global nilpotent}\label{preuveDimension}

Dans cette section, nous allons démontrer la proposition \ref{dimension}.

Comme $X$ et $D$ sont définis sur une extension de type fini sur le corps premier contenu dans $k$ ($\mathbb{Q}$ ou $\mathbb{F}_{p}$), il suffit de le faire quand $k$ est la clôture algébrique d'un corps fini $\mathbb{F}_{q}$ et quand $X$ et $D$ sont définis sur ce corps fini.

Pour tout objet $(\mathcal{E},\theta)$ de $f_{n}^{-1}(0)$, il existe une entier $1\leq s\leq n$ tel que $\theta^{s}=0$ et $\theta^{s-1}\not=0$ et on a le drapeau des images des itérées de $\theta$,
$$
\mathcal{E}_{\bullet}=(\mathcal{E}_{0}=(0)\subsetneq\mathcal{E}_{1}=\mathrm{Im}(\theta^{s-1})((1-s)D)\subsetneq\cdots\subsetneq\mathcal{E}_{s-1}=\mathrm{Im}(\theta)(-D)\subsetneq\mathcal{E}_{s}=\mathcal{E}).
$$
On peut donc associer à $(\mathcal{E},\theta)$ la partition
$$
\underline{n}(\mathcal{E},\theta)=(n_{1}+n_{2}+\cdots+n_{s}=n)
$$
définie par $n_{i}=\mathrm{rang}(\mathcal{E}_{i}/\mathcal{E}_{i-1})$ et la suite d'entiers
$$
\underline{e}(\mathcal{E},\theta)=(e_{1},\ldots,e_{s})\hbox{ avec }e_{1}+e_{2}+\cdots+e_{s}=e
$$
définie par $e_{i}=\mathrm{deg}(\mathcal{E}_{i}/\mathcal{E}_{i-1})$.
\medskip

Pour $\underline{n}$ et $\underline{e}$ fixées, les conditions $\underline{n}(\mathcal{E},\theta)=\underline{n}$ et $\underline{e}(\mathcal{E},\theta)=\underline{e}$ définissent une partie localement fermée $\mathcal{N}_{\underline{n}}^{\underline{e}}\subset f_{n}^{-1}(0)$. On obtient ainsi une stratification de $f_{n}^{-1}(0)$ et il suffit de démontrer que pour chaque strate $\mathcal{N}_{\underline{n}}^{\underline{e}}$, on a
$$
\mathrm{dim}(\mathcal{N}_{\underline{n}}^{\underline{e}})\leq d_{f_{n}}-1.
$$

À chaque $(\mathcal{E},\theta)\in\mathcal{N}_{\underline{n}}^{\underline{e}}$, on peut associer la chaîne
$$
\mathcal{F}^{\bullet}=(\mathcal{F}^{s}\rightarrow\mathcal{F}^{s-1}\rightarrow\cdots\mathcal{F}^{1})
$$
où
$$
\mathcal{F}^{i}=\mathrm{Im}(\theta^{s-i})/\mathrm{Im}(\theta^{s-i+1})(-D)=(\mathcal{E}_{i}/\mathcal{E}_{i-1})((s-i)D)
$$
est de rang $n_{i}$ et de degré
$$
f_{i}=e_{i}+n_{i}(s-i)d,
$$
et où les flèches induites par $\theta$ sont toutes génériquement surjectives. Par suite, pour que la strate $\mathcal{N}_{\underline{n}}^{\underline{e}}$ soit non vide, il faut que 
$$
n_{s}\geq n_{s-1}\geq\cdots\geq n_{1}
$$
et qu'en cas d'égalité $n_{i+1}=n_{i}$, on ait $e_{i+1}-d\leq e_{i}$ puisqu'alors $\mathcal{F}_{i+1}\rightarrow\mathcal{F}_{i}$ est une injection entre $\mathcal{O}_{X}$-modules localement libres de même rang. On suppose dans la suite que ces deux dernières conditions sont vérifiées.
\medskip

Soit $P=MN\subset\mathrm{GL}(n)$ le parabolique standard de sous-groupe de Levi $M=\mathrm{GL}(n_{1})\times \mathrm{GL}(n_{2})\times\cdots\times\mathrm{GL}(n_{s})$, et soit $\gamma$ l'élément de l'algèbre de Lie $\mathfrak{n}$ de $N$ donné par la matrice par blocs $(\gamma_{i,j})_{i,j=1,\ldots s}$ avec tous les $\gamma_{i,j}$ nuls sauf les $\gamma_{i,i+1}\in\mathrm{Mat}(n_{i}\times n_{i+1})$ qui sont égaux à $(I,0)$ où I est la matrice identité de $\mathrm{GL}(n_{i})$ et $0$ est la matrice nulle de $\mathrm{Mat}(n_{i}\times (n_{i+1}-n_{i}))$. Au point générique de la courbe $X$, on peut identifier le drapeau $\mathcal{E}_{\bullet}$ au drapeau
$$
(0)\subsetneq F^{n_{1}}\subsetneq F^{n_{1}+n_{2}}\subsetneq\cdots\subsetneq F^{n_{1}+\cdots n_{s-1}}\subsetneq F^{n}
$$
de sorte que la matrice de $\theta$ soit précisément $\gamma$.

Si on note:
\begin{itemize}
\item $P_{\gamma}$ le centralisateur de $\gamma$ dans $P$,
\item $P(\mathbb{A})^{\underline{e}}=N(\mathbb{A})M(\mathbb{A})^{\underline{e}}$
avec
$$
M(\mathbb{A})^{\underline{e}}=\{m=(g_{1},\ldots,g_{s})\in \prod_{i=1}^{s}\mathrm{GL}(n_{i},\mathbb{A})\mid \mathrm{deg}(\mathrm{det}(g_{i}))=-e_{i}\},
$$
\item $\mathbf{1}_{D}$ la fonction caractéristique de $\varpi^{-D}\mathfrak{p}(\mathcal{O})$ dans $\mathfrak{p}(\mathbb{A})$,
\item $\mathrm{d}p$ la mesure de Haar à gauche qui donne le volume $1$ à $P(\mathcal{O})$,
\end{itemize}
le  cardinal champêtre de $\mathcal{N}_{\underline{n}}^{\underline{e}}(\mathbb{F}_{q})$ est égal à
$$
|\mathcal{N}_{\underline{n}}^{\underline{e}}(\mathbb{F}_{q})|=\int_{P_{\gamma}(F)\backslash P(\mathbb{A})^{\underline{e}}}\mathbf{1}_{D}(p^{-1}\gamma p)\mathrm{d}p.
$$

On peut calculer cette intégrale en deux temps
$$
|\mathcal{N}_{\underline{n}}^{\underline{e}}(\mathbb{F}_{q})|=\int_{M_{\gamma}(F)\backslash M(\mathbb{A})^{\underline{e}}}\delta_{P}^{-1}(m)\int_{N_{\gamma}(F)\backslash N(\mathbb{A})}\mathbf{1}_{D}(m^{-1}n^{-1}\gamma nm)\mathrm{d}n\mathrm{d}m
$$
où $M_{\gamma}=M\cap P_{\gamma}$, $N_{\gamma}=N\cap P_{\gamma}$, $\mathrm{d}m$ et $\mathrm{d}n$ sont les mesures de Haar sur $M(\mathbb{A})$ et $N(\mathbb{A})$ qui donnent le volume $1$ à $M(\mathcal{O})$ et $N(\mathcal{O})$, et $\delta_{P}$ est le caractère modulaire dont la valeur sur $M(\mathbb{A})^{\underline{e}}$ est constante et égale à
$$
q^{\sum_{i<j}(n_{j}e_{i}-n_{i}e_{j})}.
$$

Maintenant pour $m$ fixé, on a
\begin{equation*}
\begin{split}
\int_{N_{\gamma}(F)\backslash N(\mathbb{A})}\mathbf{1}_{D}(m^{-1}n^{-1}\gamma nm)\mathrm{d}n = &\mathrm{vol}(N_{\gamma}(F)\backslash N_{\gamma}(\mathbb{A}),\mathrm{d}n_{\gamma})\\
 & \ \ \ \ \ \ \int_{N_{\gamma}(\mathbb{A})\backslash N(\mathbb{A})}\mathbf{1}_{D}(m^{-1}n^{-1}\gamma nm)\frac{\mathrm{d}n}{\mathrm{d}n_{\gamma}} 
\end{split}
\end{equation*}
où $\mathrm{d}n_{\gamma}$ est la mesure de Haar qui donne le volume $1$ à $N_{\gamma}(\mathcal{O})$,
et d'après Ranga~Rao, la flèche
$$
N_{\gamma}\backslash N\rightarrow \mathfrak{n}',~n\mapsto n^{-1}\gamma n-\gamma
$$
où $\mathfrak{n}'=[\mathfrak{n},\mathfrak{n}]$, est un isomorphisme algébrique qui envoie la mesure de Haar adélique $\frac{\mathrm{d}n}{\mathrm{d}n_{\gamma}}$ sur la mesure de Haar $\mathrm{d}\nu'$ sur $\mathfrak{n}'(\mathbb{A})$ qui donne le volume $1$ à $\mathfrak{n}'(\mathcal{O})$. En particulier, $N_{\gamma}$ étant un groupe unipotent de dimension $\sum_{i=1}^{s-1}n_{i}n_{i+1}$, on a
$$
\mathrm{vol}(N_{\gamma}(F)\backslash N_{\gamma}(\mathbb{A}),\mathrm{d}n_{\gamma})=q^{\sum_{i=1}^{s-1}n_{i}n_{i+1}(g-1)}.
$$

Comme
$$
\mathbf{1}_{D}(m^{-1}\gamma m+m^{-1}\nu'm)=\mathbf{1}_{D}(m^{-1}\gamma m)\mathbf{1}_{D}(m^{-1}\nu'm)
$$
et que
$$
\int_{\mathfrak{n'}(\mathbb{A})}\mathbf{1}_{D}(m^{-1}\nu'm)\mathrm{d}\nu'=q^{\sum_{i<j-1}(n_{j}e_{i}-n_{i}e_{j}+n_{i}n_{j}d)}
$$
ne dépend que de $\underline{e}$ et non de $m$, on obtient que 
$$
|\mathcal{N}_{\underline{n}}^{\underline{e}}(\mathbb{F}_{q})|=q^{\Delta}\int_{M_{\gamma}(F)\backslash M(\mathbb{A})^{\underline{e}}}\mathbf{1}_{D}(m^{-1}\gamma m)\mathrm{d}m
$$
où
\begin{equation*}
\begin{split}
\Delta=&-\sum_{i<j}(n_{j}e_{i}-n_{i}e_{j})+\sum_{i<j-1}(n_{j}e_{i}-n_{i}e_{j}+n_{i}n_{j}d)+\sum_{i=1}^{s-1}n_{i}n_{i+1}(g-1)\\
=&-\sum_{i=0}^{s-1}(n_{i+1}e_{i}-n_{i}e_{i+1})+\sum_{i<j-1}n_{i}n_{j}d+\sum_{i=1}^{s-1}n_{i}n_{i+1}(g-1).
\end{split}
\end{equation*}

La dernière intégrale
$$
\int_{M_{\gamma}(F)\backslash M(\mathbb{A})^{\underline{e}}}\mathbf{1}_{D}(m^{-1}\gamma m)\mathrm{d}m
$$
n'est autre que le cardinal champêtre de la catégorie des chaînes $\mathcal{F}^{\bullet}$. Des considérations de poids dans la cohomologie $\ell$-adique montrent alors que
$$
\mathrm{dim}(\mathcal{N}_{\underline{n}}^{\underline{e}})=\mathrm{dim}(\mathcal{C})+\Delta
$$
où $\mathcal{C}$ est le champ des chaînes $\mathcal{F}^{\bullet}$.
\medskip

D'après Garcia-Prada, Heinloth et Schmitt \cite{Garcia-Prada},
$$
\mathrm{dim}(\mathcal{C})=\sum_{i=0}^{s-1}n_{i+1}(n_{i+1}-n_{i})(g-1)+\sum_{i=0}^{s-1}(n_{i+1}f_{i}-n_{i}f_{i+1})
$$
avec la convention $n_{0}=0$ et $f_{0}=0$, soit encore
$$
\sum_{i=0}^{s-1}n_{i+1}(n_{i+1}-n_{i})(g-1)+\sum_{i=0}^{s-1}(n_{i+1}e_{i}-n_{i}e_{i+1})+\sum_{i=1}^{s-1}n_{i}n_{i+1}d
$$
avec la convention $e_{0}=0$. Par suite
$$
\mathrm{dim}(\mathcal{N}_{\underline{n}}^{\underline{e}})=\left(\sum_{i=1}^{s}n_{i}^{2}\right)(g-1)+\left(\sum_{i<j}n_{i}n_{j}\right)d.
$$

Il ne reste plus qu'à comparer cette dimension de $\mathcal{N}_{\underline{n}}^{\underline{e}}$ à 
$$
d_{f_{n}}-1=n(g-1)+\frac{n(n-1)}{2}d.
$$
Si on écrit $d=2g-2+d'$ avec $d'>0$, on a 
$$
d_{f_{n}}-1-\mathrm{dim}(\mathcal{N}_{\underline{n}}^{\underline{e}})=d'\left(\frac{n(n-1)}{2}-\sum_{i<j}n_{i}n_{j}\right)=\frac{d'}{2}\sum_{i=1}^{s}n_{i}(n_{i}-1)
$$
et cette expression est donc $\geq 0$, avec égalité si et seulement si tous les $n_{i}$ sont égaux à $1$.

\section{Le cas du diviseur canonique}\label{Canonique}

Dans le cas où $D$ est un diviseur canonique, c'est-à-dire $\mathcal{O}_{X}(D)=\Omega_{X}^{1}$, les arguments utilisés jusqu'ici ne permettent pas de contrôler le socle de $Rf_{n,\ast}^{\mathrm{st}}\mathbb{Q}_{\ell}$. En effet, on a dans ce cas
$$
d_{\mathbb{A}_{n}}=d_{f_{n}}=n^{2}(g-1)+1.
$$
et l'inégalité
$$
{1-s\geq (n-n_{1}-\cdots-n_{s})(d-2g+2)}
$$
de la fin de la démonstration du théorème \ref{ThrPr} est remplacée par
$$
0\geq (n-n_{1}-\cdots-n_{s})(2g-2-2g+2)
$$
qui ne sert à rien. Néanmoins, avec les notations de cette démonstration, pour $a\in\mathbb{A}_{n,\lambda}\cap \mathrm{Socle}(R_{f_{n,\ast}}^{\mathrm{st}}\mathbb{Q}_{\ell})$, en combinant les inégalités
$$
d_{a}\geq n_{a}^{+}(Rf_{n,\ast}^{\mathrm{st}}\mathbb{Q}_{\ell})\geq\mathrm{d}_{a}^{\mathrm{ab}}(J_{n})
$$
et
$$
\mathrm{d}_{a_{i}}^{\mathrm{ab}}(J_{n_{i}})\geq d_{a_{i}},~\forall i=1,\ldots,s,
$$
à l'égalité
$$
d_{a}=d_{a_{1}}+\cdots+d_{a_{s}},
$$
on obtient les égalités
$$
d_{a}=n_{a}^{+}(Rf_{n,\ast}^{\mathrm{st}}\mathbb{Q}_{\ell})=\mathrm{d}_{a}^{\mathrm{ab}}(J_{n})
$$
et la première de ces égalités implique que $Rf_{n,\ast}^{\mathrm{st},2d_{f_{n}}}\mathbb{Q}_{\ell}$ admet comme facteur direct un faisceau $\ell$-adique non trivial de support $\overline{\{a\}}$.

Maintenant, on peut espérer que $Rf_{n,\ast}^{\mathrm{st},2d_{f_{n}}}\mathbb{Q}_{\ell}$ est localement constant sur chaque strate $\mathbb{A}_{n,\lambda}$ comme c'est le cas sur les strates contenues dans l'ouvert génériquement régulier semi-simple, c'est-à-dire celles avec $\lambda$ de la forme $((n_{1},\ldots,n_{s}),(1,\ldots,1))$  (voir \cite{Chaudouard}), et sur la strate nilpotente, c'est-à-dire celle avec $\lambda=((1),(n))$, sur laquelle $\mathcal{N}_{n}^{e}$ et $N_{n}^{e,\mathrm{st}}$ sont constants (cette strate est isomorphe à $H^{0}(X,\Omega_{X}^{1})$ et on identifie la fibre en $a$ à la fibre en $0$ par $(\mathcal{E},\theta)\mapsto(\mathcal{E},\theta-a)$).

Dans le cas où $D$ est un diviseur canonique, on peut donc espérer démontrer que le socle de $Rf_{n,\ast}^{\mathrm{st}}\mathbb{Q}_{\ell}$ est contenu dans l'ensemble fini des points génériques des strates $\mathbb{A}_{n,\lambda}$.

Pour $n=2$, c'est bien le cas puisque la seule strate non contenue dans l'ouvert génériquement régulier semi-simple est la strate nilpotente.

\bibliographystyle{plain}

\begin{flushleft}
Pierre-Henri Chaudouard \\
Université Paris Diderot (Paris 7)\\ 
Institut de Mathématiques de Jussieu-Paris Rive Gauche \\ 
UMR 7586 \\
Bâtiment Sophie Germain, Case 7012 \\
F-75205 PARIS Cedex 13, France \\
\url{Pierre-Henri.Chaudouard@imj-prg.fr} \\
\bigskip

Gérard Laumon \\
CNRS et Université Paris-Sud \\
UMR 8628 \\
Mathématique, Bâtiment 425 \\
F-91405 Orsay Cedex , France \\
\url{gerard.laumon@math.u-psud.fr}\\
\end{flushleft}
\end{document}